\documentclass[12pt,a4paper]{article}
\usepackage{indentfirst}
\usepackage{amsmath,amssymb,amsfonts,amsthm,graphics}
\usepackage{makeidx}
\usepackage{color}

\newtheorem{theorem}{Theorem}

\newtheorem{proposition}{Proposition}

\newtheorem{observation}{Observation}

\begin{document}
\title{\Large\bf Note on minimally $k$-rainbow connected graphs\footnote{Supported by NSFC
No.11071130 and `the Fundamental Research Funds for the Central
Universities".}}
\author{\small Hengzhe Li, Xueliang Li, Yuefang Sun, Yan Zhao
\\
\small Center for Combinatorics and LPMC-TJKLC
\\
\small Nankai University, Tianjin 300071, China
\\
\small lhz2010@mail.nankai.edu.cn; lxl@nankai.edu.cn;\\
\small bruceseun@gmail.com; zhaoyan2010@mail.nankai.edu.cn}
\date{}
\maketitle
\begin{abstract}
An edge-colored graph $G$, where adjacent edges may have the same
color, is {\it rainbow connected} if every two vertices of $G$ are
connected by a path whose edge has distinct colors. A graph $G$ is
{\it $k$-rainbow connected} if one can use $k$ colors to make $G$
rainbow connected. For integers $n$ and $d$ let $t(n,d)$ denote the
minimum size (number of edges) in $k$-rainbow connected graphs of
order $n$. Schiermeyer got some exact values and upper bounds for
$t(n,d)$. However, he did not get a lower bound of $t(n,d)$ for
$3\leq d<\lceil\frac{n}{2}\rceil $. In this paper, we improve his
lower bound of $t(n,2)$, and get a lower bound of $t(n,d)$ for
$3\leq d<\lceil\frac{n}{2}\rceil$.

\end{abstract}

\section{Introduction}

A communication network consists of nodes and links connecting them.
In order to prevent hackers, one can set a password in each link. To
facilitate the management, one can require that the number of
passwords is small enough such that every two nodes can exchange
information by a sequence of links which have different passwords.
This problem can be modeled by a graph and studied by means of
rainbow connection.

All graphs in this paper are undirected, finite and simple. We refer
to book \cite{bondy} for notation and terminology not described
here. A path in an edge-colored graph $G$, where adjacent edges may
have the same color, is a $rainbow\ path$ if no pair of edges are
colored the same. An edge-coloring of $G$ with $k$ colors is a {\it
$k$-rainbow connected coloring} if every two distinct vertices of
$G$ are connected by a rainbow path. A graph $G$ is {\it $k$-rainbow
connected} if $G$ has a $k$-rainbow connected coloring. Note that
Schiermeyer used the term {\it rainbow $k$-connected in
\cite{sch2}}. However we think that it is better to use the term
{\it $k$-rainbow connected} since this will distinguish it from the
term {\it rainbow $k$-connectivity}, which means that there are many
rainbow paths between every pair of vertices, see \cite{char2}. The
rainbow connection number $rc(G)$ of $G$ is the minimum integer $k$
such that $G$ has a $k$-rainbow connected coloring. It is easy to
see that $rc(G)\geq diam(G)$ for any connected graph $G$, where
$diam(G)$ is the diameter of $G$.

The rainbow connection number was introduced by Chartrand et al. in
\cite{char}. Bounds on the rainbow connection numbers of graphs have
been studies in terms of other graph parameters, such as radius,
dominating number, minimum degree, connectivity, etc., see
\cite{bas,chan,char,kri,lili,lisun,sch}. In \cite{chak}, Chakraborty
et al. investigated the hardness and algorithms for the rainbow
connection number, and showed that given a graph $G$, deciding if
$rc(G)=2$ is NP-complete. In particular, computing $rc(G)$ is
NP-hard.

For integers $n$ and $d$ let $t(n,d)$ denote the minimum size
(number of edges) in $k$-rainbow connected graphs of order $n$.
Because a network which satisfies our requirements and has as less
links as possible can cut costs, reduce the construction period and
simplify later maintenance, the study of this parameter is
significant. Schiermeyer \cite{sch2} mainly investigated the upper
bound of $t(n,d)$ and showed the following results.

\begin{theorem}{\upshape (Schiermeyer\cite{sch2})}

$(i)$ $t(n,1)={n\choose{2}}$.

$(ii)$ $t(n,2)\leq (n+1)\lfloor\log_2 n\rfloor-2^{\lfloor\log_2
n\rfloor}-2$.

$(iii)$ $ t(n,3))\leq 2n-5$.

$(iv)$  For $4\leq d <\frac{n-1}{2}$, $t(n,d)\leq
n-1+\left\lceil\frac{n-2}{d-2}\right\rceil.$

$(v)$ For $\frac{n}{2}\leq d\leq n-2$, $t(n,d)=n$.

$(vi)$ $t(n,n-1)=n-1$.
\end{theorem}

In \cite{sch2}, Schiermeyer also got a lower bound of $t(n,2)$ by an
indirect method, and showed that $t(n,2)\geq n\log_2
n-4n\log_2\log_2 n-5n$ for sufficiently large $n$. Nevertheless, he
did not get a lower bound of $t(n,d)$ for $3\leq
d<\lceil\frac{n}{2}\rceil $. In this paper, we use a different
method to improve his lower bound of $t(n,2)$, and moreover we get a
lower bound of $t(n,d)$ for $3\leq d<\lceil\frac{n}{2}\rceil $.

\section{Main results}

Let $G$ be a $2$-rainbow connected graph of order $n$ with maximum
degree $\Delta(G)$. Pick a vertex $u\in V(G)$. Since $d(u)\leq
\Delta(G)$, there exist at most $\Delta(G)$ vertices adjacent to
$u$, and at most $\Delta(G)(\Delta(G)-1)$ vertices at distance $2$
from $u$. Since $diam(G)\leq rc(G)\leq 2$, we derive $n\leq
1+\Delta(G)+\Delta(G)(\Delta(G)-1)$. Thus, $\Delta(G) \geq
\sqrt{n-1}$. Since $\Delta(G)$ is an integer, we get

\begin{equation}
 \Delta(G) \geq \left\lceil\sqrt{n-1}\;\right\rceil.
\end{equation}

Next, we investigate the lower bound of $e_2(n)$.

\begin{proposition} For sufficiently large $n$,
$t(n,2)\geq n\log_2n-4n\log_2\log_2 n-2n.$
\end{proposition}

\begin{proof}
Let $G$ be a graph with diameter $2$ and $c$ be a $2$-rainbow
connected coloring of $G$ with colors blue and red. Set
$k=\lfloor\log_2 n\rfloor^2-1$ and denote by $S$ the set of vertices
with degrees less than $k$. Assume that $S=\{u_1,u_2,\ldots,u_s\}$
and $T=V(G)\backslash S=\{u_{s+1},u_{s+2},\ldots,u_{s+t}\}$, where
$s+t=n$. For sufficiently large $n$, $k=\lfloor\log_2
n\rfloor^2-1\linebreak[2]\leq
\left\lceil\sqrt{n-1}\;\right\rceil\leq \Delta(G)$. By $(1)$ we know
that $T$ is nonempty. If $t=|T|\geq \frac{2n}{\log_2 n}$, then

\begin{eqnarray*}
e(G)\geq \frac{1}{2}\sum_{v\in T}d_{G}(v)&\geq&\frac{2n}{2\log_2 n}
\left(\left\lfloor\log_2 n\right\rfloor^2-1\right)\\[6pt]
&\geq&\frac{n}{\log_2 n}
\left(\left(\log_2 n-1\right)^2-1\right)\\[6pt]
&=& n\log_2 n-2n, \end{eqnarray*}

\noindent and we are done.

Suppose $t<\frac{2n}{\log_2 n}$, that is, $s>n-\frac{2n}{\log_2 n}$.
It is sufficient to show that $e(S,T)\geq n\log_2n-4n\log_2\log_2
n-2n.$

For every $u_i,\, 1\leq i\leq s$, we define a vector as follows:

$$\alpha(u_i)=(b_{i,1},b_{i,2},\ldots,b_{i,t}),$$ where
  \begin{equation*} b_{i,j}=
  \begin{cases} 1, & if\ c(u_iu_{s+j})\ is\ red;\\
  -1, & if\ c(u_iu_{s+j})\ is\ blue;\\
  0,  & if\ u_i\ and\ u_{s+j}\ is\ nonadjacent.
  \end{cases}
  \end{equation*}

Suppose $|N(u_i)\cap T|=a_i$, where $1\leq i\leq s$. Then
$e(S,T)=\sum_{i=1}^{s}a_i$, where $e(S,T)$ denotes the number of
edges between $S$ and $T$. We now estimate the value of $e(S,T)$.
For each $\alpha(u_i)$, we define a set $B_i$ as follows:
$B_i=\{$vectors obtained from $\alpha(u_i)$ by replace ``$0$'' of
$\alpha(u_i)$ by ``$1$'' or ``$-1$''$\}$. Because $|N(u_i)\cap
T|=a_i$, we have $|B_i|=2^{t-a_i}$ for each $i$, where $1\leq i\leq
s$. Set $B=\bigcup_{i=1}^s B_i$. Then $B$ is a multiset of
$t$-dimensional vectors with elements $1$ and $-1$. For each
$\alpha\in B$, $n_{\alpha}$ denotes the number of $\alpha$'s in $B$.
We have the following claim.\vspace{6pt}

{\it Claim~1. For each $\alpha\in B$, $n_{\alpha}\leq
k^2+1$.}\vspace{6pt}

\noindent{\em Proof of Claim~$1$.} If Claim~$1$ is not true, that
is, there exists a vector $\alpha$, without loss of generality, say
$\alpha=(b_1,b_2,\ldots,b_t)$, such that $n_{\alpha}\geq k^2+2$.
Clearly, it is not possible that there exists some $B_i$ such that
$B_i$ contains two $\alpha$'s. Thus, there exist $k^2+2$ integers,
without loss of generality, say $1,2,\ldots,k^2+2$, such that
$\alpha\in B_i$, where $1\leq i\leq k^2+2$. we next show that for
each $i$, where $2\leq i\leq k^2+2$, the distance between $u_1$ and
$u_i$ in $G[S]$ is at most $2$. In fact,
$c(u_1u_{s+j})=b_j=c(u_iu_{s+j})$ follows from the definition of
$B_1$ and $B_i$. Thus there exists no rainbow path between $u_1$ and
$u_i$ through a vertex of $T$. Hence, there must exist a rainbow
path between $u_1$ and $u_i$ with length at most $2$ in $G[S]$. On
the other hand, since $\Delta(G[S])\leq k$, the number of vertices
at distance $2$ from $u_1$ is at most $k^2+1$, which is a
contradiction, and the claim is thus true.

By Claim~$1$ we know
$$\sum_{i=1}^s|B_i|\leq (k^2+1)2^t,$$
Since $|B_i|=2^{t-a_i}$ for each $i$, where $1\leq i\leq s$,
$$\sum_{i=1}^s 2^{-a_i}\leq (k^2+1).$$
By the inequality between the geometrical and arithmetical means, we
have

$${\large\sqrt[s]{2^{-e(S,T)}}}={\large\sqrt[s]{2^{-\sum_{i}^s a_i}}}\leq
\frac{1}{s}\sum_{i=1}^s 2^{-a_i}\leq \frac{k^2+1}{s}.$$

\noindent Using the log function on both sides, we get
\begin{eqnarray}
e(S,T)\geq s\log_2 s-s\log_2 (k^2+1).
\end{eqnarray}
Note that $k^2+1=(\lfloor\log_2 n\rfloor^2-1)^2+1\leq((\log_2
n)^2-1)^2+1\leq \left(\log_2 n\right)^4-2\left(\log_2
n\right)^2+2\leq \left(\log_2 n\right)^4$. We have
\begin{eqnarray}
e(S,T)\geq s\log_2 s-4s\log_2\log_2 n.
\end{eqnarray}
Since $e(S,T)$ is monotonically increasing in $s$ and
$s>n-\frac{2n}{\log_2 n}$, we have
\begin{eqnarray*}
e(S,T)&\geq& \left(n-\frac{2n}{\log_2 n}\right)\log_2
\left(n-\frac{2n}{\log_2 n}\right)-4\left(n-\frac{2n}{\log_2
n}\right)\log_2\log_2 n\\[6pt]
\end{eqnarray*}
\begin{eqnarray*}
&=&n\log_2\left(n-\frac{2n}{\log_2  n}\right)-\frac{2n}{\log_2
n} \log_2\left(n-\frac{2n}{\log_2  n}\right)\\[6pt]
&&-4n\log_2\log_2 n+\frac{8n}{\log_2 n}\log_2\log_2 n\\[6pt]
&=&n\log_2n+n\log_2\left(1-\frac{2}{\log_2n}\right)-2n-\frac{2n}{\log_2
n}\log_2\left(1-\frac{2}{\log_2  n}\right)\\[6pt]
&&-4n\log_2\log_2 n+\frac{8n}{\log_2 n}\log_2\log_2 n\\[6pt]
&\geq&n\log_2n-4n\log_2\log_2 n-2n\\[6pt]
&&+n\log_2\left(1-\frac{2}{\log_2n}\right)+\frac{8n}{\log_2 n}\log_2\log_2 n\\[6pt]
&=&n\log_2n-4n\log_2\log_2 n-2n\\[6pt]
&&+\frac{2n}{\log_2 n}\log_2\left(\left(1-\frac{2}{\log_2n}\right)^{\frac{\log_2 n}{2}}
\left(\log_2 n\right)^4\right)\\[6pt]
&\geq&n\log_2n-4n\log_2\log_2 n-2n.
\end{eqnarray*}
The last inequality holds since
$\left(1-\frac{2}{\log_2n}\right)^{\frac{\log_2 n}{2}}$ is
monotonically increasing in $n$ and tends to $\frac{1}{e}$. This
completes the proof.
\end{proof}

Before showing the lower bound on $t(n,d)$ for each $d\geq 3$, we
need the following theorem and obeservation.

\begin{theorem}{\upshape (Jarry and Laugier\cite{jarry})}

Any $2$-edge-connected graph of order $n$ and of odd diameter $p\geq
2$ contains at least $\left\lceil\frac{np-(2p+1)}{p-1}\right\rceil$
edges. Any $2$-edge-connected graph of order $n$ and of even
diameter $p$ contains at least
$\min\left\{\left\lceil\frac{np-(2p+1)}{p-1}\right\rceil,
\left\lceil\frac{(n-1)(p+1)}{p}\right\rceil\right\}$
edges.\end{theorem}

Note that $\left\lceil\frac{
np-(2p+1)}{p-1}\right\rceil=\left\lceil\frac{(n-1)(p-1)+n+p-1-(2p+1)}{p-1}\right\rceil=\left\lceil
n-1+\frac{n-p-2}{p-1}\right\rceil\geq\left\lceil
n-1+\frac{n-p-2}{p}\right\rceil=
n-2+\left\lceil\frac{n-2}{p}\right\rceil$ and
$\left\lceil\frac{(n-1)(p+1)}{p}\right\rceil=\left\lceil\frac{(n-1)p+n-1}{p}\right\rceil=\left\lceil
n-1+\frac{n-1}{p}\right\rceil\linebreak[2]\geq\left\lceil
n-1+\frac{n-p-2}{p}\right\rceil=
n-2+\left\lceil\frac{n-2}{p}\right\rceil$. Thus, any
2-edge-connected graph of order $n$ and of diameter $p\geq 2$
contains at least $n-2+\left\lceil\frac{n-2}{p}\right\rceil$ edges.

Let $G$ be a graph and $c$ be a rainbow connected coloring of $G$.
It is easy to see that different bridges of $G$ must receive
different coloring under $c$. Therefore, the following observation
is obvious.

\begin{observation}
The rainbow connection number of a graph is at least the number of
bridges in the graph.
\end{observation}

\begin{proposition}For $3\leq d<\left\lceil\frac{n}{2}\right\rceil$,
$$t(n,d)\geq n-d-3+\left\lceil\frac{n-1}{d}\right\rceil.$$
\end{proposition}
\begin{proof}
Let $G$ be a $k$-rainbow connected graph of order $n$. Suppose that
$G$ has $k$ bridges and $G'$ is the graph obtained from $G$ by
deleting all the bridges. Then $G'$ has $k+1$ components. We have
$k\leq d$ by Observation~1. Suppose that $G_1,G_2,\cdots,G_{k_1}$
are the nontrivial components of $G'$. Thus, $G'$ has $k_2=k+1-k_1$
trivial components. Let $n_i$ denote the order of $G_i$ and $d_i$
denote the diameter of $G_i$. We have that $n_1+n_2+\cdots
+n_{k_1}=n-k_2$.\vspace{6pt}

{\it Claim~2. Each of the graph $G_i$ is either a $2$-edge-connected
graph with diameter at least $2$ or a complete graph of order at
least $3$.}\vspace{6pt}

\noindent{\em Proof of Claim~$2$.} Suppose that some of the graph
$G_i$ is neither a 2-edge-connected graph with diameter at least
$2$, nor a complete graph of order at least $3$. That is, the graph
$G_i$ is a complete graph of order less than $3$. Since $G_i$ is
nontrivial, $G_i$ is a complete graph of order~$2$. However, the
only edge of $G_i$ is clearly a cut edge of $G$, a contradiction.
Thus, this claim holds.

Now consider the number of edges of $G_i$. If $G_i$ is a
2-edge-connected graph with diameter $d_i\geq 2$, then by Theorem~4
we have that $e(G_i)\geq
n_i-2+\left\lceil\frac{n_i-2}{d_i}\right\rceil$. If not, that is,
$G_i$ is a complete graph of order at least $3$ by Claim~2. We have
that $e(G_i)\geq {n_i\choose{2}}\geq
n_i-2+\left\lceil\frac{n_i-2}{d_i}\right\rceil$. Thus, $e(G_i) \geq
n_i-2+\left\lceil\frac{n_i-2}{d_i}\right\rceil$ for each $i$, where
$1\leq i\leq k_1$.\vspace{6pt}

{\it Claim~3. $d_i\leq d$.}\vspace{6pt}

\noindent{\em Proof of Claim~$3$.} Let $x$ and $y$ be two vertices
in $G_i$. Since the shortest path connecting $x$ and $y$ in $G$ must
be a path contained $G_i$, we have $d_{G_i}(x,y)\leq d_{G}(x,y)$.
Thus, $d_i\leq diam(G)\leq d$.

Now evaluate the number of edges in $G$. We have
\begin{eqnarray*}
e(G)&=&k+\sum_{i=1}^{k_1}e(G_i)\\[6pt]
&=&k+\sum_{i=1}^{k_1} \left(n_i-2+\left\lceil\frac{n_i-2}{d_i}\right\rceil\right)\\[6pt]
\end{eqnarray*}
\begin{eqnarray*}
&\geq&k+\sum_{i=1}^{k_1} \left(n_i-2+\left\lceil\frac{n_i-2}{d}\right\rceil\right)\\[6pt]
&\geq&k+\sum_{i=1}^{k_1} \left(n_i-2+\frac{n_i-2}{d}\right)\\[6pt]
&=& k+(n-k_2)-2k_1+\frac{(n-k_2)-2k_1}{d}\\[6pt]
&=& n-1-k_1+\frac{n-k-1-k_1}{d}\\[6pt]
&\geq & n-1-k+\frac{n-2k-1}{d}\\[6pt]
&\geq& n-d-1+\frac{n-2d-1}{d}\\[6pt]
&=& n-d-3+\frac{n-1}{d}.
\end{eqnarray*}
Thus, $e_d(n)\geq n-d-3+\lceil\frac{n-1}{d}\rceil$.
\end{proof}

\begin{proposition}
For $3\leq d <\lceil\frac{n}{2}\rceil$, $t(n,d)\leq
n-2+\lceil\frac{n}{d-1}\rceil$.
\end{proposition}
\begin{proof} Set $n=q(d-1)+r$, where $0< r\leq d-1$. Then
$q=\lceil\frac{n}{d-1}\rceil-1$. Pick $q$ cycles of length $d$, say
$C_1,C_2,\cdots C_q$. Identify $u_i$ as a vertex $u$, where $u_i\in
V(C_i)$ and $i=1,2,\ldots,q$. Finally, we attach $r$ edges to the
new vertex $u$. Denote by $G_d(n)$ the resulting graph, see Fig.1
for details. It is easy to check that
$e(G_d(n))=n-1+q=n-2+\lceil\frac{n}{d-1}\rceil$.

\begin{figure}[h,t,b,p]
\begin{center}
\scalebox{0.7}[0.7]{\includegraphics{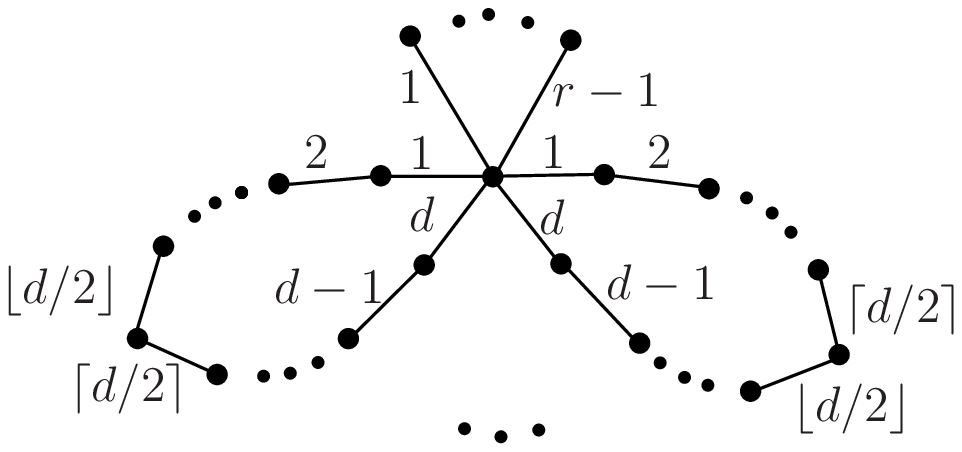}}\vspace{1.0cm}

Fig.1. A $d$-rainbow connected coloring of the graph $G_d(n)$.
\end{center}
\end{figure}

Now, we show that the graph $G_d(n)$ is $d$-rainbow connected. For
each $C_i$, first, color an incident edge of $u$ on $C_i$ by $1$,
and the other one by $d$; second, color the edge adjacent to a
$1$-color edge on $C_i$ by $2$, and color the edge adjacent to a
$d$-color edge on $C_i$ by $d-1$. We do this until all edges of
$C_i$ are colored. For the $r$ bridges, we color them by different
colors that have been used to color $C_i$. It is easy to check that
the above coloring is a $k$-rainbow connected coloring.
\end{proof}

Combining Propositions~1, 2 and~3, the following theorem holds.

\begin{theorem}
$(i)$ For sufficiently large $n$, $t(n,2)\geq
n\log_2n-4n\log_2\log_2 n-2n$.

$(ii)$ For $3\leq d <\lceil\frac{n}{2}\rceil$,
$n-d-3+\left\lceil\frac{n-1}{d}\right\rceil\leq t(n,d)\leq
n-2+\left\lceil\frac{n}{d-1}\right\rceil.$
\end{theorem}

\end{document}